\documentclass[12pt]{amsart}

\usepackage{latexsym,amsmath,amssymb,amsthm,amsfonts}

\usepackage[numeric]{amsrefs}

\usepackage{hyperref}

\usepackage[capitalise]{cleveref}

\usepackage{multido,pst-plot,pstricks,pst-node}
\renewcommand{\S}{{\mathbb{S}}}
\newcommand{\eps}{\varepsilon}

\newcommand{\E}{{\mathbb{E}}}

\newcommand{\K}{{\mathcal{K}}}
\newcommand{\J}{{\mathcal{J}}}

\newcommand{\erf}{{\rm erf}}

\newcommand{\intr}{{\rm int}}

\renewcommand{\phi}{{\boldsymbol{\varphi}}}

\newcommand{\qtq}[1]{{\quad\text{#1}\quad}}

\newtheorem{theorem}{Theorem}

\newtheorem{proposition}[theorem]{Proposition}

\newtheorem{corollary}[theorem]{Corollary}

\theoremstyle{remark}
\newtheorem{remark}[theorem]{Remark}

\title{On Hadwiger's covering problem in small dimensions}

\author{A. Arman}
\address{Department of Mathematics, University of Manitoba, Winnipeg, MB, R3T 2N2, Canada}
\email{andrew0arman@gmail.com}
\thanks{The first author was supported by the University of Manitoba Research Grant Program project ``Covering problems with applications in Euclidean Ramsey theory and convex geometry'' and PIMS PDF}

\author{A.\ Bondarenko}
\address{Department of Mathematical Sciences, Norwegian University of Science and Technology, NO-7491 Trondheim, Norway}
\email{andriybond@gmail.com}
\thanks{The second author was supported in part by Grant 275113 of the Research Council of Norway.}

\author{A.\ Prymak}
\address{Department of Mathematics, University of Manitoba, Winnipeg, MB, R3T 2N2, Canada}
\email{prymak@gmail.com}
\thanks{Corresponding author: Andriy Prymak; email: {\tt prymak@gmail.com} \\ \phantom{fff} The third author was supported by NSERC of Canada Discovery Grant RGPIN-2020-05357.}



\begin{document}

\begin{abstract}
	
Let $H_n$ be the minimal number such that any $n$-dimensional convex body can be covered by $H_n$ translates of interior of that body. Similarly $H_n^s$ is the corresponding quantity for symmetric bodies. It is possible to define $H_n$ and $H_n^s$ in terms of illumination of the boundary of the body using external light sources, and the famous Hadwiger's covering conjecture (illumination conjecture) states that $H_n=H_{n}^s=2^n$. 

In this note we obtain new upper bounds on $H_n$ and $H_{n}^s$ for small dimensions $n$. Our main idea is to cover the body by translates of John's ellipsoid (the inscribed ellipsoid of the largest volume). Using specific lattice coverings, estimates of quermassintegrals for convex bodies in John's position, and calculations of mean widths of regular simplexes, we prove the following new upper bounds on $H_n$ and $H_n^s$: $H_5\le 933$, $H_6\le 6137$, $H_7\le 41377$, $H_8\le 284096$, $H_4^s\le 72$, $H_5^s\le 305$, and $H_6^s\le 1292$.

For larger $n$, we describe how the general asymptotic bounds $H_n\le \binom{2n}{n}n(\ln n+\ln\ln n+5)$ and $H_n^s\le 2^n n(\ln n+\ln\ln n+5)$ due to Rogers and Shephard can be improved for specific values of $n$.
\end{abstract}

\maketitle

\section{Introduction}

Let $\K_n$ be the family of all convex bodies in $\E^n$, i.e., all convex compact sets $K\subset\E^n$ with a nonempty interior ($\intr(K)\neq \emptyset$). For $A,B\subset \E^n$, we denote by
\[
C(A,B):=\min \left\{N:\exists t_1,\dots,t_N\in\E^n \text{ satisfying } A\subset \bigcup_{j=1}^N(t_j+B)\right\},
\]
the minimal number of translates of $B$ needed to cover $A$.

Hadwiger~\cite{Ha} raised the question of determining the value of
\[
H_n=\min \{C(K,\intr(K)):K\in\K_n\}
\]
for all $n\ge 3$. 

Considering an $n$-cube, one immediately sees that $H_n\ge 2^n$, and the well-known Hadwiger's covering conjecture states that $H_n=2^n$ for $n\ge 3$ with equality only for parallelepipeds. It was shown that $H_2=4$ by Levi~\cite{Le}. Aside from Levi and Hadwiger, the conjecture may be associated with the names of Boltyanski, who in~\cite{Bo} has established an equivalent formulation in terms of illumination of the boundary of the body by external light sources, and with Gohberg-Markus~\cite{Go-Ma} who asked the question in terms of the minimal number of smaller homothetic copies of $K$ required to cover $K$. As of today, the conjecture, which is also known as the illumination conjecture, is wide open. For details about the history and partial results for special classes of convex bodies see, e.g., \cite{Be}. 

In what follows we will outline the current approaches in obtaining the upper bounds on $H_n$ and our modifications allowing to obtain new bounds.  

To this end, let  $\K_n^s$ is the subfamily of $\K_n$ consisting of centrally symmetric convex bodies, and define $$H_n^s=\min \{C(K,\intr(K)):K\in\K_n^s\}.$$ 

The best known \emph{explicit} upper bound on $H_n$ in high dimensions follows from the results of Rogers and Zong~\cite{Ro-Zo}:
\begin{equation}\label{eqn:asymptotic}
	H_n\le \binom{2n}{n}n(\ln n+\ln\ln n+5).
\end{equation}
The asymptotic behavior of $\binom{2n}{n}$ is $\frac{4^n}{\sqrt{2\pi n}}$, so this upper bound is  $(4+o(1))^n$ while the conjecture states $H_{n}=2^n$. Remarkable sub-exponential improvements of~\eqref{eqn:asymptotic} have been obtained only recently by Huang, Slomka, Tkocz and Vritsiou~\cite{HSTV}. Namely, using ``thin-shell'' volume estimates they obtained
\begin{equation}\label{eqn:HSTV}
	H_n\le \exp(-c\sqrt{n}) 4^n.
\end{equation}
This was later improved, using the new breakthrough bound on isotropic constant, by Campos, van Hintum, Morris and Tiba~\cite{CHMT} to
\begin{equation}\label{eqn:CHMT}
	H_n\le \exp\Bigl(\frac{-cn}{\ln^8 n}\Bigr)4^n.
\end{equation}
Both constants $c$ in~\eqref{eqn:HSTV} and~\eqref{eqn:CHMT} are independent of $n$ but not given explicitly. For centrally symmetric bodies, one has
\begin{equation}\label{eqn:sym-asymp}
	H_n^s\le 2^n n(\ln n+\ln\ln n+5),
\end{equation} 
which is asymptotically close to the conjectured $2^n$ (see~\cite{Ro-Zo}). 

The first step in obtaining any of the inequalities~\eqref{eqn:asymptotic}--\eqref{eqn:sym-asymp} is the result of Rogers~\cite{Ro} on covering density of $\E^n$ by translates of a body $L$. Namely, let \begin{equation*}
\theta(L)=\lim_{R\to\infty}\inf_{\Lambda}\left\{ 
\frac{|L|\#\Lambda}{(2R)^n}: [-R,R]^n\subset \bigcup_{\lambda\in\Lambda} (x_\lambda+L)
\right\}
\end{equation*}
denote the covering density of space by translates of $L$, where $|L|$ is the volume of $L$, and $\#\Lambda$ is the cardinality of $\Lambda$. Rogers~\cite{Ro} showed that for any body $L\in \K_n$,
\begin{equation}\label{eqn:rogers}
\theta(L)\le n(\ln n+\ln\ln n+5).	
\end{equation}

The second step is the use of the inequality
\begin{equation}\label{eqn:rogerszong}
	C(K,L)\le \frac{|K-L|}{|L|}\theta(L),
\end{equation}
which is valid for any $K,L\in\K^n$ and was proved by Rogers and Zong~\cite{Ro-Zo}. 

For symmetric bodies, one obtains~\eqref{eqn:sym-asymp} from~\eqref{eqn:rogerszong} by choosing $L=(1-\eps)K$ with $\eps\to0+$ and using the fact that $|K-K|=|2K|=2^n|K|$ for $K\in\K_n^s$. For general bodies, \eqref{eqn:asymptotic} follows from~\eqref{eqn:rogerszong} by applying the Rogers-Shephard (see~\cite{Ro-Sh}) inequality $|K-K|\le \binom{2n}{n}|K|$ for $K\in \K_n$. 

The improvements~\eqref{eqn:HSTV} and~\eqref{eqn:CHMT} are both based on choosing $L$ to be the largest (by volume) centrally symmetric subset of $K$. With such choice, we have $|K-L|/|L|\le |2K|/|L|\le 2^n \Delta_{KB}(K)^{-1}$, where
\[
\Delta_{KB}(K):=\max_{x\in \E^n}\frac{|K\cap (x-K)|}{|K|}
\]
is the K\"ovner-Besicovitch measure of symmetry of $K$. Simple averaging gives $\Delta_{KB}(K)\ge 2^{-n}$ for any $K\in \K_n$, while~\cite{HSTV} and~\cite{CHMT} obtain the corresponding sub-exponential improvements to the lower bound on $\Delta_{KB}(K)$,  yielding~\eqref{eqn:HSTV} and~\eqref{eqn:CHMT}.

For low dimensions, Lassak~\cite{La} showed that
\begin{equation}\label{eqn:lassak}
	H_n\le (n+1)n^{n-1}-(n-1)(n-2)^{n-1},
\end{equation}
which outperforms~\eqref{eqn:asymptotic} for $n\le 5$. For $n=3$ this gives $H_3\le 34$, which was improved to $H_3\le20$ by Lassak~\cite{La2}, then to $H_3\le 16$ by Papadoperakis~\cite{Pa}, and then to $H_3\le14$ by Prymak~\cite{Pr}. For slightly larger dimensions, it was shown in~\cite{Pr-Sh} that $H_4\le 96$, $H_5\le 1091$ and $H_6\le 15373$ improving both~\eqref{eqn:asymptotic} and~\eqref{eqn:lassak} for $n=4,5,6$. Then Diao~\cite{D} obtained $H_5\le 1002$ and $H_6\le 14140$. All of these results in low dimensions were based on comparing the body with a suitable parallelepiped. For the symmetric case, Lassak~\cite{La3ds} obtained the sharp result $H_3^s=8$, but for $n\ge 4$ no estimate better than~\eqref{eqn:sym-asymp} or than the corresponding bound on $H_n$ (obviously, $H_n^s\le H_n$) was known.

Our main result is the following new bounds.
\begin{theorem}\label{thm:5678}
	$H_5\le 933$, $H_6\le 6137$, $H_7\le 41377$, $H_8\le 284096$, $H_4^s\le 72$, $H_5^s\le 305$, and $H_6^s\le 1292$.
\end{theorem}
The main idea of the proof is to utilize~\eqref{eqn:rogerszong} with $L$ being the maximal (by volume) inscribed ellipsoid into $K$ (such ellipsoids were characterized by John~\cite{J}). After an appropriate affine transform, we can assume that $L$ is the unit ball $B_2^n$ in $\E^n$, while $K$ is in the so-called John's position, so applying~\eqref{eqn:rogerszong} we get
\begin{equation}\label{eq:ourbound}C(K, \intr{K})\leq \frac{|K+B_2^n|}{|B_2^n|}\theta(B_2^n).\end{equation}

Next we make use of several geometric results which allow us to obtain an upper bound on $|K+B_2^n|$. One ingredient in such estimates is the fact that the mean width and the volume of a body in John's position are largest for the regular simplex (general case) or for the cube (symmetric case), the results due to Ball~\cite{Ba}, Barthe~\cite{Bar}, Schechtman and Schmuckenschl\"ager~\cite{Sc-S}. Another ingredient in our estimates is a Bonnnesen-type inequality by Bokowski and Heil~\cite{Bo-He} on quermassintegrals of $K$. Finally, we use upper bounds on $\theta(B_2^n)$ for specific small $n$ which arise from known lattice coverings. 

\cref{thm:5678} is proved in Sections~\ref{sec:prel}--\ref{sec:proof1}. Of possibly independent interest are estimates of the mean width of regular simplex in dimensions $5\le n\le 8$ given in \cref{sec:mean width simplex}. 

In \cref{sec:rogers opt} we show how one can improve~\eqref{eqn:rogers} for each fixed $n$ by optimizing choices of certain parameters in the original proof of Rogers~\cite{Ro} (the original proof provides a succinct bound valid for all $n$). Consequently, \eqref{eqn:asymptotic} and \eqref{eqn:sym-asymp} can be somewhat improved for $n$ larger than those covered by \cref{thm:5678}. 

To finalize, in~\cref{tbl:hn,tbl:hns} we provide the best known upper bounds on $H_n$ and $H_n^s$ for $3\le n\le 14$.

\begin{table}[h!]
	\begin{center}
		\begin{tabular}{|c|c|c||c|c|c|}
			\hline
			$n$ & $H_n\le$ & reference & $n$ & $H_n\le$ & reference  \\
			\hline
			\hline
			$3$ & $14$ & \cite{Pr} & $9$ & $2\,064\,332$ & Prop.~\ref{cor:hn and hns} \\
			\hline
			$4$ & $96$ & \cite{Pr-Sh} & $10$ & $8\,950\,599$ & Prop.~\ref{cor:hn and hns} \\
			\hline
			$5$ & $933$ & Th.~\ref{thm:5678} &$11$ & $38\,482\,394$ & Prop.~\ref{cor:hn and hns}\\
			\hline
			$6$ & $6\,137$ & Th.~\ref{thm:5678} & $12$ & $164\,319\,569$ & Prop.~\ref{cor:hn and hns}\\
			\hline
			$7$ & $41\,377$ & Th.~\ref{thm:5678}& $13$ & $697\,656\,132$ & Prop.~\ref{cor:hn and hns}\\
			\hline
			$8$ & $284\,096$ & Th.~\ref{thm:5678}& $14$ & $2\,947\,865\,482$ & Prop.~\ref{cor:hn and hns}\\
			\hline
		\end{tabular}
	\end{center}
	\caption{Best known upper bounds on $H_n$ for $3\le n\le 14$.}
	\label{tbl:hn}
\end{table}

\begin{table}[h!]
	\begin{center}
		\begin{tabular}{|c|c|c||c|c|c|}
			\hline
			$n$ & $H_n^s\le$ & reference & $n$ & $H_n^s\le$ & reference  \\
			\hline
			\hline
			$3$ & $8$ & \cite{La3ds} & $9$ & $21\,738$ & Prop.~\ref{cor:hn and hns} \\
			\hline
			$4$ & $72$ & Th.~\ref{thm:5678} & $10$ & $49\,608$ & Prop.~\ref{cor:hn and hns} \\
			\hline
			$5$ & $305$ & Th.~\ref{thm:5678} &$11$ & $111\,721$ & Prop.~\ref{cor:hn and hns}\\
			\hline
			$6$ & $1\,292$ & Th.~\ref{thm:5678} & $12$ & $248\,895$ & Prop.~\ref{cor:hn and hns}\\
			\hline
			$7$ & $3\,954$ & Prop.~\ref{cor:hn and hns}& $13$ & $549\,506$ & Prop.~\ref{cor:hn and hns}\\
			\hline
			$8$ & $9\,370$ & Prop.~\ref{cor:hn and hns}& $14$ & $1\,203\,936$ & Prop.~\ref{cor:hn and hns}\\
			\hline
		\end{tabular}
	\end{center}
	\caption{Best known upper bounds on $H_n^s$ for $3\le n\le 14$.}
	\label{tbl:hns}
\end{table}

\section{Preliminaries}\label{sec:prel}
In what follows, $B_2^{n}=\{x\in \E^{n} \;: \|x\|\leq 1\}$, and $\S^{n-1}=\{x\in \E^{n} \;: \|x\|= 1\}.$ The $n$-dimensional volume is denoted by $|\cdot |_n$, and subscript is usually dropped if the value of $n$ is clear from the context.

\subsection{John's position}
Let $\J_n$ be the family of convex bodies $K\in \K_n$ which are in John's position, i.e. $B_2^n$ is the maximal volume ellipsoid of $K$ (see, e.g.,~\cite{isotropic-book}*{Sect.~1.5.1}). Similarly, define $\J_n^s$ to be the family of centrally symmetric convex bodies $K\in \K_n^s$ which are in John's position. For any $K\in \K_n$ (or $K\in\K_n^s$) there exists an affine image of $K$ in $\J_n$ (or in $\J_n^s$). 

John's theorem  (see e.g. \cite{aga-book}*{Thm.~2.1.3, Remark 2.1.17}) implies that 
\begin{equation}\label{eqn:John outer balls}
K\subset nB_2^n \qtq{for any} K\in\J_n, \qtq{and} K\subset \sqrt{n} B_2^n \qtq{for any} K\in \J_n^s.	
\end{equation}

\subsection{Quermassintegrals}
In order to bound $H_n$ and $H_{n}^{s}$ we will use the inequality~\eqref{eq:ourbound}, and so we start by discussing upper bounds on $|K+B_2^n|$, where $K$ is in John's position. 

By the Steiner's formula, for any $K\in\K_n$
\begin{equation}\label{eqn:steiner}
	|K+tB_2^n|=\sum_{j=0}^n \binom nj W_j(K) t^j,
\end{equation}
where 
\[
W_j(K)=V(\underset{n-j \text{ times}}{\underbrace{K,\dots,K}}, \underset{j \text{ times}}{\underbrace{B_2^n,\dots,B_2^n}})
\]
is the $j$-th quermassintegral of $K$ and $V(\cdot)$ is the mixed volume, see, e.g., \cite{aga-book}*{Sect.~1.1.5} or~\cite{isotropic-book}*{Sect.~1.4.2}. Additionally, $W_0(K)=|K|$, $W_1(K)=\partial (K)/n$, where $\partial(K)$ is the surface area of $K$, and $W_n(K)=|B_2^n|$. Also note that if $B_2^n\subset K$, then due to the monotonicity of mixed volumes we have $W_i(K)\ge W_j(K)$ for $i\le j$.

Next, we will discuss upper bounds on $W_i(K)$ for $K\in \J_n$ and $K\in \J_{n}^s$ respectively. 

Let $T^n$ be a regular simplex in $\E^n$ of unit edge length, and let $\Delta^n$ be a dilation of $T^n$ for which $B_2^n$ is the inscribed ball, then $\Delta^{n}=\sqrt{2(n+1)n} \; T^{n}$. Also, let $C^n:=[-1,1]^n$ be a cube circumscribed about $B^2_n$. Ball~\cite{Ba} proved that among all convex bodies in $\E^n$ simplexes have maximal volume ratio (volume ratio measures how much of the volume of the whole body can be contained in the largest inscribed ellipsoid) while for the symmetric ones such a maximizer is the cube~\cite{Ba-sym}. These results can be stated in our terms as follows:
\begin{equation}\label{eqn:volume}
	W_0(K)\le \begin{cases}
		W_0(\Delta^n), &\text{if } K\in\J_n,\\
		W_0(C^n), &\text{if } K\in\J_n^s.
	\end{cases}
\end{equation}

For $K\in\K_n$ and a direction $u\in\S^{n-1}$, the support function is defined as
$
h_K(u)=\sup\{\langle x,y\rangle: y\in K\}
$. Let $\sigma$ be the rotationally invariant probability measure on $\S^{n-1}$, then the mean width of $K$ is defined by
\[
w(K)=\int_{\S^{n-1}} h_K(u)\,d\sigma(u)
.\footnote{Note that sometimes the mean width is defined as $\int_{\S^{n-1}} h_K(u)+h_{K}(-u)\,d\sigma(u)
	$.} \]

We have $W_{n-1}(K)=|B_2^n|w(K)$, which is a partial case of the Kubota's formula~\cite{aga-book}*{eq (1.1.1)}. Barthe~\cite{Bar}*{Thm. 3} proved~\footnote{The result of Barthe is stated in terms of $\ell$-norm, and can be translated in the language of mean width using the formula~\cite{Bar}*{p. 685} for the $\ell$-norm of the dual body.} that among the bodies from $\J_n$, the mean width is maximized for $\Delta^n$, while Schechtman and Schmuckenschl\"ager~\cite{Sc-S} (the proof is also included in~\cite{Bar}*{Thm. 2}) remarked that $C^n$ maximizes the mean width among the bodies from $\J_n^s$, i.e.,
\begin{equation}\label{eqn:mean width}
	W_{n-1}(K)\le \begin{cases}
		W_{n-1}(\Delta^n), &\text{if } K\in\J_n,\\
		W_{n-1}(C^n), &\text{if } K\in\J_n^s.
	\end{cases}
\end{equation}
We estimate $w(\Delta^n)$ for the required values of $n$ in the next section. It is known~\cite{Fi}, \cite{DCG-book}*{Sec.~13.2.3} (or can be obtained by a straightforward computation) that $w(C^n)=2\frac{|B_2^{n-1}|_{n-1}}{|B_{2}^{n}|_n}$  for $n\geq 2$, and so $W_{n-1}(C^{n})=2|B_{2}^{n-1}|$.

%

\begin{remark}\label{remark:surface}
	The results~\cite{Ba}*{Thm. 1, Thm. $1^\prime$} imply that $W_1(K)\le W_1(\Delta^n)$ for any $K\in\J_n$ (recall that $W_1(K)=\partial(K)/n$, where $\partial(K)$ is the surface area of $K$). This would not lead to any improvements in our context as $W_1(\Delta^n)=W_0(\Delta^n)$ and we get the same upper bound on $W_1(K)$ as from $W_1(K)\le W_0(K)$. A similar remark regarding the inequality $W_1(K)\leq W_1(C^{n})$ also holds for $K \in \J^s_n$, and follows from~\cite{Ba}*{Thm. 2},~\cite{Ba-sym}*{Thm. 3}.
\end{remark}

\begin{remark}
	It is natural to conjecture that $W_j(K)\le W_j(\Delta^n)$ (and that $W_j(K)\le W_j(C^n)$) for any $K\in\J_n$ (respectively, $K\in\J_n^s$) and all $0\le j\le n$. This is an obvious equality for $j=n$ and is valid for $j\in\{0,1,n-1\}$ as described above (\eqref{eqn:volume}, \cref{remark:surface}, \eqref{eqn:mean width}). 
\end{remark}

Finally, we need a Bonnesen-style inequality by Bokowski and Heil~\cite{Bo-He}. If $K\in\K_n$ satisfies $K\subset R B_2^n$, then for all $0\le i<j<k\le n$
\begin{equation}\label{eqn:bonnesen}
	W_j(K) \le \frac{(k-j)(i+1)R^iW_i(K)+(j-i)(k+1)R^kW_k(K)}{(k-i)(j+1)R^j}=:B_{R,i,j,k}(W_i(K),W_k(K)).
\end{equation}
Recall that in our settings we can choose $R$ according to~\eqref{eqn:John outer balls}.

\subsection{Density of coverings by balls} One of the approaches to construct specific efficient coverings of the space by balls is to use lattices, see~\cite{Co-Sl}*{Ch.~2}. Considering the $A_n^{*}$ lattice yields the following estimate (valid for all $n$):
\begin{equation}\label{eqn:covering density}
	\theta(B_2^n)\le |B_2^n|\sqrt{n+1}\left(\frac{n(n+2)}{12(n+1)}\right)^{n/2}.
\end{equation}
It turns out that the above is optimal (smallest possible \emph{lattice} covering density) for $2\le n\le 5$ and provides the best known upper bound on $\theta(B_2^n)$ in most dimensions $10\le n\le 21$. Better lattices were found by Sch\"urmann and Vallentin~\cite{SchVall} for $6\le n\le 8$, which is important for our applications. To the best of our knowledge, there have been no improvements after the work~\cite{SchVall}, where an interested reader can find a brief survey of the topic. We list the corresponding lattice covering densities of $\E^n$ by balls in \cref{tbl:covdens}.
\begin{table}[h!]
	\begin{center}
		\begin{tabular}{|c|c||c|c|}
			\hline
			$n$ & least known lattice covering density & $n$ & least known lattice covering density  \\
			\hline
			\hline
			$2$ & $1.209199$ & $8$ & $3.142202$ \\
			\hline
			$3$ & $1.463505$ & $9$ & $4.340185$ \\
			\hline
			$4$ & $1.765529$ & $10$ & $5.251713$ \\
			\hline
			$5$ & $2.124286$ & $11$ & $5.598338$ \\
			\hline
			$6$ & $2.464801$ & $12$ & $7.510113$ \\
			\hline
			$7$ & $2.900024$ & $13$ & $7.864060$ \\
			\hline
		\end{tabular}
	\end{center}
	\caption{Least known lattice covering densities, as in~\cite{SchVall}.}
	\label{tbl:covdens}
\end{table}

\section{Estimates on mean width of regular simplex}\label{sec:mean width simplex}

The values of $2w(T^n)$ for $2\le n\le 6$ expressed as certain integrals and numerically evaluated with high precision can be found in~\cite{Fi}. We remark that in~\cite{Fi} expected values of widths of simplexes are considered, while the mean width as we defined here (and as commonly defined in the literature on geometry, see~\cite{aga-book,isotropic-book}) is the expected value of the support function, this discrepancy results in the mean width of $T^{n}$ from~\cite{Fi} (and~\cite{Su}) being equal to $2w(T^{n})$. We need upper estimates of $w(T^n)$ for $n=7,8$, which we were unable to find in the literature. We present a simple computational technique to obtain upper and lower estimates on $w(T^n)$ which will suffice for our purposes. 

The starting point is the following representation that follows the work~\cite{Fi} by Finch and references therein, in particular, Sun~\cite{Su}. 

Let $ F(x)=\frac{1}{\sqrt{2\pi}}\int_{-\infty}^{x}e^{-t^2/2}\; dt$ be the comulative distribution function of the standard normal distribution, then $F(x)=\frac{1}{2}+\frac{1}{2}\erf(\frac{x}{\sqrt{2}})$, where $\erf(z)=\frac2{\sqrt{\pi}}\int_0^ze^{-t^2}\,dt$ is the error function. 

For $n\geq 0$ let
\begin{align*}
g_{n+1}(x)&:=1-F(x)^{n+1}-(1-F(x))^{n+1}= 1-\left(\frac{1+\erf(\tfrac{x}{\sqrt{2}})}2\right)^{n+1}-\left(\frac{1-\erf(\tfrac{x}{\sqrt{2}})}2\right)^{n+1}.
\end{align*}
The following formula for mean width follows from~\cite{Su} and was derived in~\cite{Fi}:
\begin{align}\label{eqn:mean width integral}
w(T^n)= & \frac{\Gamma(\tfrac n2)}{2\Gamma(\tfrac {n+1}2)}\int_0^\infty g_{n+1}(x)\,dx. 
\end{align}
Next, we estimate the integral in the formula.
\begin{proposition}\label{prop:mean width estimates}
	For any $a>2$ and any positive integers $n$, $N$, the following estimates hold:
	\begin{equation}\label{eqn:bound with integral sums}
		\frac aN\sum_{k=1}^N g_{n+1}\left(\tfrac{ak}N\right)
		\le \int_0^\infty g_{n+1}(x)\,dx \le
		 \frac aN
		 \sum_{k=0}^{N-1} g_{n+1}\left(\tfrac{ak}N\right)
		 +\frac{n+1}{\sqrt{2\pi}} \exp\left(-a\right)		 .
	\end{equation}
\end{proposition}
\begin{proof}
We directly estimate the ``tail'' of the integral and use simple endpoint Riemann sums for the ``main'' part of the integral in~\eqref{eqn:mean width integral}.
	
Recalling that $\frac{d}{dx}F(x)=\frac{1}{\sqrt{2\pi}}e^{-x^2/2}$ it is straightforward to verify that $g_{n+1}$ is a positive strictly decreasing function on $[0,\infty)$. Hence, considering the upper and the lower Riemann sums for $\int_0^a g_{n+1}(x)\,dx$ and the uniform partition of $[0,a]$ into $N$ subintervals, we obtain
\begin{equation}\label{eqn:integral sums}
			\frac aN\sum_{k=1}^N g_{n+1}\left(\tfrac{a\cdot k}N\right)
	\le \int_0^a g_{n+1}(x)\,dx \le
	\frac aN
	\sum_{k=0}^{N-1} g_{n+1}\left(\tfrac{a\cdot k}N\right).
\end{equation}
For any $x>a\geq 2$
\begin{align*}
	g_{n+1}(x) &\le 1-F(x)^{n+1} 
	= 1-\left(1-F(-x)\right)^{n+1}=1-\left(1-\frac{1}{\sqrt{2\pi}}\int_{x}^{\infty}e^{-t^2/2}\,dt\right)^{n+1} \\
	&\le \frac{n+1}{\sqrt{2\pi}}\int_{x}^\infty e^{-t^2/2}\,dt
	\le \frac{n+1}{\sqrt{2\pi}}\int_{x}^\infty e^{-t}\,dt
	=\frac{n+1}{\sqrt{2\pi}}\exp\left(-x \right),
\end{align*}
so
\begin{equation*}
	\int_a^\infty g_{n+1}(x)\,dx \le \frac{n+1}{\sqrt{2\pi}} \exp\left(-a\right).
\end{equation*}
Taking this inequality and the evident $\int_a^\infty g_{n+1}(x)\,dx\ge0$ into account, we deduce~\eqref{eqn:bound with integral sums} from~\eqref{eqn:mean width integral} and~\eqref{eqn:integral sums}.
\end{proof}

Since the error function can be computed numerically with any given precision, employing a simple SageMath (\cite{sagemath}) computation~\cite{github}, we obtain the following corollary.
\begin{corollary}\label{cor:mw}
	The following inequalities hold: 
	\begin{align*}
		0.4208 \le & w(T^5) \le 0.4215, \\
		0.4067 \le & w(T^6) \le 0.407, \\ 
		0.39425 \le & w(T^7) \le 0.39427, \\ 
		0.383 \le & w(T^8) \le 0.38301.
	\end{align*}
\end{corollary}
These estimates for $n=5,6$ are consistent with the values obtained in~\cite{Fi}; we include them here for completeness. While our computational method allows to obtain a tighter gap in the estimates, we only derived what was necessary for our application of estimating $H_n$ from above ensuring that no further improvement is possible (even if the value of the lower bound on $w(T^n)$ is used, the upper bound on the integer value $H_n$ does not improve). The computations take less than two hours on a modern personal computer.

\section{Proof of \cref{thm:5678}}\label{sec:proof1}

Let us begin with the general (not necessarily symmetric) case. Since $C(K,\intr(K))$ is invariant under affine transforms, we can assume that $K\in \J^n$. Then using $B_2^n\subset K$, compactness and~\eqref{eqn:rogerszong}, we have
\begin{align*}
	C(K,\intr(K))&\le C(K,\intr(B_2^n))\le \lim_{r\to1^-} C(K,rB_2^n) \\
	&\le \lim_{r\to1^-} \frac{|K-rB_2^n|}{|rB_2^n|}\theta(B_2^n)
	= \frac{|K+B_2^n|}{|B_2^n|}\theta(B_2^n).
\end{align*} 
Hence, applying~\eqref{eqn:steiner},
\begin{equation}\label{eqn:est with wj}
	C(K,\intr(K))\le \frac{\theta(B_2^n)}{|B_2^n|} \sum_{j=0}^n \binom nj W_j(K).
\end{equation}
The upper bounds for $\theta(B_2^n)$ come from Table~\ref{tbl:covdens},
so it remains to estimate $W_i(K)$. As $\Delta^n=\sqrt{2n(n+1)}T^n$,  by~\eqref{eqn:volume}, 
\begin{equation}\label{eqn:est vol}
	W_i(K)\le W_0(K)\le W_0(\Delta^n)=(2n(n+1))^{n/2}|T^n|=\frac{n^{n/2}(n+1)^{(n+1)/2}}{n!}
\end{equation}
for any $0\le i\le n$. Using~\eqref{eqn:mean width},
\begin{equation}\label{eqn:est mw}
	W_{n-1}(K)\le W_{n-1}(\Delta^n)=|B_2^n|w(\Delta^n)=|B_2^n|\sqrt{2n(n+1)}w(T^n).
\end{equation}
Trivially, $W_n(K)=|B_2^n|$.

Next we combine the above, and use~\eqref{eqn:bonnesen} (valid with $R=n$ due to~\eqref{eqn:John outer balls}) for appropriate parameters. The calculations were performed in the script~\cite{github}.

{\bf Bound on $H_5$.} We apply~\eqref{eqn:est with wj} and estimate the quermassintegrals as follows. For $0\le i\le 2$, use~\eqref{eqn:est vol}; \eqref{eqn:est mw} and~\eqref{eqn:bonnesen} give $W_4(K)\le W_4(\Delta^5)$ and $W_3(K)\le B_{5,2,3,4}(W_0(\Delta^5),W_4(\Delta^5))$; recall that $W_5(K)=|B_2^5|$. Combining the above, using \cref{cor:mw} and calculating the actual value of the bound yields $H_5\le 933$.

{\bf Bound on $H_6$.} The arguments are similar to the previous case. The required application of \eqref{eqn:bonnesen} is $W_4(K)\le B_{6,3,4,5}(W_0(\Delta^6),W_5(\Delta^6))$. The result is $H_6\le 6137$. 

{\bf Bounds on $H_n$ for $n=7,8$.} We proceed similarly with the only difference that $W_j(K)\le B_{n,n-4,j,n-1}(W_0(\Delta^n),W_{n-1}(\Delta^n))$ is used for $j\in\{n-3,n-2\}$ yielding $H_7\le 41377$ and $H_8\le 284096$.

When the body is centrally symmetric, we follow a similar route, with the following differences. In place of~\eqref{eqn:est vol} we have
\begin{equation}\label{eqn:est vol cube}
	W_i(K)\le W_0(K)\le W_0(C^n)=2^n
\end{equation}
for any $0\le i\le n$. Using~\eqref{eqn:mean width} we get 
\begin{equation}\label{eqn:est W_n-1 cube}
W_{n-1}(K)\le W_{n-1}(C^{n})=2|B_2^{n-1}|.
\end{equation}
Finally, by~\eqref{eqn:John outer balls}, the inequality~\eqref{eqn:bonnesen} can be used with $R=\sqrt{n}$.

{\bf Bound on $H_4^s$.} Apply~\eqref{eqn:est with wj} and bound the quermassintegrals as follows. For $i=0,1$ use~\eqref{eqn:est vol cube}; use~\eqref{eqn:est W_n-1 cube} to get $W_{3}(K)\leq2|B_{2}^{3}|$; we have $W_4(K)=W_4(C^{4})=|B_2^4|$;  by~\eqref{eqn:bonnesen} $W_2(K)\le B_{2,1,2,3}(W_1(C^4),W_3(C^4))$. Combining the above and calculating the value of the bound yields $H_4^s\le 72$. 

{\bf Bounds on $H_n^s$ for $n=5,6$.} We use similar arguments: \eqref{eqn:est vol cube} for $i=0,1,2$; $W_n=W_n(C^n)=|B_{2}^{n}|$; \eqref{eqn:est W_n-1 cube} for $i=n-1$;  and $W_j(K)\le B_{\sqrt{n},2,j,n-1}(W_0(C^n),W_{n-1}(C^n))$ for $3\le j\le n-2$, which imply $H_5^s\le 305$ and $H_6^s\le 1292$.

\begin{remark}
	In the above computations, we used the exact value for the density arising from $A_n^*$ lattice from the estimate~\eqref{eqn:covering density} for $n=4,5$. For $n=6,7,8$ we used the values of covering densities from~\cite{SchVall} presented in~\cref{tbl:covdens} increased by $5\cdot 10^{-6}$ as they were given to 6 decimal places, while we require an upper bound. Even if such an increase is not performed, the resulting \emph{integer} valued upper bounds in \cref{thm:5678} would not change, so the accuracy given in~\cite{SchVall} is more than sufficient for our needs.
\end{remark}

\section{Estimates via optimized Rogers bound}\label{sec:rogers opt}

\begin{proposition}\label{cor:hn and hns}
	Suppose $n\ge 3$. If
\begin{equation}\label{eqn:rog function}
	r_n=\min_{x\in(0,1/n)} f_n(x), \qtq{where} f_n(x)=(1+x)^n(1-n\ln(x)),
\end{equation}	
then
\begin{equation}\label{eqn:rog cors}
	\theta(K)\le r_n \qtq{for any} K\in \K_n,
	\quad H_n\le \binom{2n}{n} r_n, \qtq{and}
	H_n^s\le 2^n r_n.
\end{equation}
\end{proposition}
\begin{proof}
	The first inequality in~\eqref{eqn:rog cors} is established in~\cite{Ro}*{p.~5}. (The bound~\eqref{eqn:rogers} was obtained in~\cite{Ro} by taking $x=\frac1{n\ln n}$ in~\eqref{eqn:rog function}.) The other two inequalities in~\eqref{eqn:rog cors} follow from~\eqref{eqn:rogerszong} in the same way as~\eqref{eqn:asymptotic} and~\eqref{eqn:sym-asymp} do. 
\end{proof}

For our purposes, it suffices to use the straightforward $r_n\le \min\{ f_n(\frac{j}{Nn}):1\le j\le N-1\}$ with $N=1000$. The resulting bounds on $r_n$ for $3\le n\le 14$ are given in~\cref{tbl:Rogers} (the computed values of $\min\{ f_n(\frac{j}{Nn}):1\le j\le N-1\}$ are rounded up in the sixth digit, so they represent actual upper bounds), and the estimates on $H_n$ and $H_n^s$ can be found in \cref{tbl:hn,tbl:hns}. For the computations, see~\cite{github}. We remark that $\max\{\theta(K),\ K\in\K_2\}=\frac32$ was established by F\'ary~\cite{Fary}. 

	\begin{table}[h!]
	\begin{center}
		\begin{tabular}{|c|c||c|c|}
			\hline
			$n$ & $\max\{\theta(K),\ K\in\K_n\}\le$ & $n$ & $\max\{\theta(K),\ K\in\K_n\}\le$  \\
			\hline
			\hline
			$3$ & 10.064123 & $9$ & 42.458503 \\
			\hline
			$4$ & 14.916986 & $10$ & 48.445515 \\
			\hline
			$5$ & 20.024359 & $11$ & 54.551530 \\
			\hline
			$6$ & 25.362768 & $12$ & 60.765566 \\
			\hline
			$7$ & 30.898293 & $13$ & 67.078451 \\
			\hline
			$8$ & 36.603890 & $14$ & 73.482436 \\
			\hline
		\end{tabular}
	\end{center}
	\caption{Upper bounds on $\max\{\theta(K),\ K\in\K_n\}$, for $3\le n\le 14$.}
	\label{tbl:Rogers}
\end{table}

\end{document}